\documentclass[a4paper,11pt]{amsart}

\usepackage{latexsym}
\usepackage{amssymb}
\usepackage{amsthm}
\usepackage{amscd}
\usepackage{amsmath}
\usepackage{graphics}
\usepackage[all]{xy}

\newtheorem{Theorem}{Theorem}[section]
\newtheorem{Lemma}[Theorem]{Lemma}
\newtheorem{Corollary}[Theorem]{Corollary}
\newtheorem{Proposition}[Theorem]{Proposition}
\newtheorem{Remark}[Theorem]{Remark}

\newtheorem{Definition}[Theorem]{Definition}

\newtheorem{Question}[Theorem]{Question}

\def\fkp{{\frak p}}

\def\fkm{{\frak m}}
\def\fkM{{\frak M}}
\def\fkn{{\frak n}}

\def\opn#1#2{\def#1{\operatorname{#2}}}

\opn\Spec{Spec}
\opn\Supp{Supp}
\opn\supp{supp}
\opn\Max{Max}
\opn\max{max}
\opn\Min{Min}
\opn\min{min}
\opn\Ass{Ass}
\opn\Assh{Assh}
\opn\Ann{Ann}
\opn\depth{depth}
\opn\rank{rank}
\opn\Proj{Proj}

\opn\Mat{Mat}

\opn\Tot{Tot}

\opn\Sym{Sym}
\def\Rees{{\mathcal R}}

\def\bf{\boldsymbol}

\opn\div{div}
\opn\Div{Div}
\opn\cl{cl}
\opn\Cl{Cl}

\opn\Ker{Ker}
\opn\Coker{Coker}
\opn\Im{Im}
\opn\Hom{Hom}
\opn\Tor{Tor}
\opn\Ext{Ext}
\opn\End{End}
\opn\Fitt{Fitt}
\opn\Aut{Aut}
\opn\id{id}

\opn\nat{nat}
\opn\pff{pf}
\opn\Pf{Pf}
\opn\GL{GL}
\opn\SL{SL}
\opn\G{G}
\opn\E{E}
\opn\H{H}
\opn\M{M}
\opn\mod{mod}
\opn\ord{ord}
\opn\det{det}
\opn\Soc{Soc}

\opn\chara{char}
\opn\length{\ell}
\opn\pd{pd}
\opn\rk{rk}
\opn\projdim{proj\,dim}
\opn\injdim{inj\,dim}
\opn\rank{rank}
\opn\depth{depth}
\opn\grade{grade}
\opn\height{ht}
\opn\embdim{emb\,dim}
\opn\codim{codim}

\renewcommand{\bar}{\overline}

\title{Asymptotic vanishing of homogeneous components of multigraded modules and its applications}
\author{Futoshi Hayasaka}
\address{Department of Environmental and Mathematical Sciences, Okayama University, 
3-1-1, Tsushimanaka, Kita-ku, Okayama, 700-8530, JAPAN}
\email{hayasaka@okayama-u.ac.jp}

\begin{document}

\maketitle

\begin{abstract}
In this article, we give a condition on the vanishing of finitely many homogeneous components
which must imply the asymptotic vanishing for multigraded modules. 
We apply our result to multi-Rees algebras of ideals. 
As a consequence, we obtain a result on normality of monomial ideals, which extends and improves 
the results of Reid, Roberts and Vitulli \cite{RRV}, Singla \cite{S}, and Sarkar and Verma \cite{SV}. 
\end{abstract}

\section{Introduction}

Let $(A, \fkm)$ be a Noetherian local ring with the maximal ideal $\fkm$ 
and let $R$ be a Noetherian standard $\mathbb N^r$-graded ring with $R_{\bf 0}=A$, i.e., 
$R$ is generated in degrees $\bf e_1, \dots , \bf e_r$ over $A$ with $R_{\bf e_i} \neq (0)$ for all $i=1, \dots , r$. 
Let $M$ be a finitely generated $\mathbb Z^r$-graded $R$-module and let $N$ be a graded $R$-submodule of $M$. 
In this article, we study the asymptotic vanishing property of homogeneous components $[M/N]_{\bf n}$ of the quotient $M/N$. Particularly, we give a condition on the vanishing.

Our main motivation comes from the following facts on the normality of monomial ideals. Let $S=k[X_1, \dots , X_d]$ be a polynomial ring 
over a field $k$ and let $I$ be a monomial ideal in $S$. In 2003,  Reid, Roberts and Vitulli  \cite{RRV} proved that if $I, I^2, \dots , I^{d-1}$ are 
integrally closed, then all the powers of $I$ are integrally closed, i.e., $I$ is a normal ideal. In 2007, Singla improved this result. Let 
$\ell =\lambda (I)$ be the analytic spread of $I$. Then Singla \cite{S} proved that $I$ is a normal ideal if $I, I^2, \dots , I^{\ell-1}$ are integrally closed. 
Note that the inequality $\ell \leq d$ always holds true. On the other hand, more recently, 
Sarkar and Verma \cite{SV} gave the other generalization of the original result of  Reid, Roberts and Vitulli  \cite{RRV}. 
Let $I_1, \dots , I_r$ be $(X_1, \dots , X_d)$-primary
monomial ideals in $S$. They proved that if $\bf I^{\bf n}=I_1^{n_1} \cdots I_r^{n_r}$ is integrally closed for any $\bf n \in \mathbb N^r$ 
with $1 \leq |\bf n| \leq d-1$, where 
$|\bf n|=n_1+\dots +n_r$ denotes the sum of all entries of a vector $\bf n=(n_1, \dots, n_r) \in \mathbb N^r$, 
then all the power products $\bf I^{\bf n}$ are integrally closed.   
Thus, the original result of  Reid, Roberts and Vitulli  \cite{RRV} was partially generalized to finitely many monomial ideals. 
The approach in \cite{SV} is different from the ones in \cite{RRV, S}. In \cite{SV}, the authors used the local cohomology modules of 
the multi-Rees algebra $\Rees(\bf I)$ of ideals $I_1, \dots , I_r$ 
and its normalization $\bar \Rees(\bf I)$ instead of convex geometry as in \cite{RRV, S}. 

The purpose of this article is to extend and improve all of the above results 
with a more general approach inspired by \cite{SV}. 
In order to state our results, 
let us recall some notations associated to multigraded modules. 

For a Noetherian standard $\mathbb N^r$-graded ring $R$ with $R_{\bf 0}=A$, let $R_{++}=R_{\bf e}R$ be the irrelevant ideal of 
$R$ where $\bf e=\bf e_1 + \dots +\bf e_r=(1, \dots , 1) \in \mathbb N^r$. 
Let $R_+=(R_{\bf e_1}, \dots , R_{\bf e_r})R$. Note that $R_{++}=R_+$ when $r=1$. 
Let $\frak M=\frak m R+R_+$ be the homogeneous maximal ideal of $R$. For a finitely generated $R$-module $M$, we set 
$$\Supp_{++}(M)=\{ P \in \Spec R \mid M_P \neq (0), P \ \mbox{is graded and} \ R_{++} \nsubseteq P \}. $$
Then the spread $s(M)$ of $M$ introduced by Kirby and Rees \cite{KR} is defined to be  
$$s(M)=\dim \Supp_{++}(M/\fkm M)+1. $$ 
Here we set $\dim \emptyset =-1$. Note that for an ideal $I$ in $A$, the spread $s(\Rees(I))$ of the Rees algebra of $I$ is just the analytic spread 
$\lambda(I)$ of the ideal $I$. 
For any $i=1, \dots , r$, let 
$$a^i(M)=\sup \{ k \in \mathbb Z \mid [H^{\dim M}_{\frak M}(M)]_{\bf n} \neq (0) \ \mbox{for some} \ \bf n \in \mathbb Z^r \ \mbox{with} \ n_i=k \},  $$
where $H^i_{\frak M}(M)$ is the $i$th local cohomology module of $M$ with respect to $\frak M$ in the category of $\mathbb Z^r$-graded $R$-modules. 
Then the $\bf a$-invariant vector is defined to be 
$$\bf a(M)=(a^1(M), \dots , a^r(M)) \in \mathbb Z^r. $$
Then our main result can be stated as follows.  

\begin{Theorem}\label{main}
Let $R$ be a Noetherian standard $\mathbb N^r$-graded ring such that $R_{\bf 0}=A$ is a local ring. 
Let $M$ be a finitely generated $\mathbb Z^r$-graded $R$-module and let $N$ be a graded $R$-submodule of $M$. 
Suppose that $M$ is a Cohen-Macaulay graded $R$-module. 
Let $\ell \geq |\bf a(M)|+s(M)+r-1$ and assume that 
$$[M/N]_{\bf n}=(0) \ \mbox{for all} \ \bf n \geq \bf a(M) + \bf e \ \mbox{with} \ |\bf n| = \ell. $$ 
Then we have that  
$$[M/N]_{\bf n}=(0) \ \mbox{for all} \ \bf n \geq \bf a(M) + \bf e \ \mbox{with} \ |\bf n| \geq \ell. $$
\end{Theorem}

The most interesting case is the case where $R$ is the multi-Rees algebra of ideals in $A$ and $M$ is its normalization. 
Let $I_1, \dots , I_r$ be ideals in $A$. We set 
$$\Rees(\bf I)=A[I_1t_1, \dots , I_rt_r]$$
the multi-Rees algebra of the ideals $I_1, \dots , I_r$ and
$$\bar \Rees(\bf I)=\sum_{\bf n \in \mathbb N^r} \bar{\bf I^{\bf n}} \bf t^{\bf n}$$ 
the multi-Rees algebra of its integral closure filtration, 
where $t_1, \dots , t_r$ are indeterminates. 
By applying Theorem \ref{main} to this special case, we have the following. 

\setcounter{section}{5}
\setcounter{Theorem}{5}

\begin{Theorem}\label{multiRees}
Let $(A, \fkm)$ be a Noetherian local ring of $\dim A=d>0$. 
Let $I_1, \dots , I_r$ be ideals in $A$ such that $\dim \Rees(\bf I)=d+r$. Suppose that
$\Rees(\bf I) \subseteq \bar \Rees(\bf I)$ is module-finite and $\bar \Rees(\bf I)$ is Cohen-Macaulay. 
Let $\ell \geq \lambda(I_1 \cdots I_r)-1$ and assume that 
$$\bf I^{\bf n} \ \mbox{is integrally closed for all} \ \bf n \in \mathbb N^r \ \mbox{with} \ |\bf n| = \ell. $$ 
Then we have that 
$$\bf I^{\bf n} \ \mbox{is integrally closed for all} \ \bf n \in \mathbb N^r \ \mbox{with} \ |\bf n| \geq \ell. $$ 

In particular, if 
$$\bf I^{\bf n} \ \mbox{is integrally closed for any} \ \bf n \in \mathbb N^r \ \mbox{with} \ 0 \leq |\bf n| \leq \lambda(I_1 \cdots I_r)-1, $$
then 
all the power products $\bf I^{\bf n}$ of $I_1, \dots , I_r$ are integrally closed.   
\end{Theorem}

This is related to the classic Zariski's theory on integrally closed ideals in a regular local ring of dimension two. 
Let $I$ and $J$ be integrally closed ideals in a two-dimensional regular local ring $A$. 
Then Zariski proved in \cite[Appendix 5]{ZS} that the product $IJ$ is also integrally closed. 
Lipman and Teissier \cite{LT} 
proved that the reduction number of integrally closed ideals 
is always at most one. 
Huneke and Sally \cite{HS} proved that the Rees algebra of integrally closed ideals is always Cohen-Macaulay. 
In this sense, Theorem \ref{multiRees} can be viewed as an analogue of the classic theorem of 
Zariski in higher dimensional regular local rings.

Let $I_1, \dots , I_r$ be monomial ideals in a polynomial ring $S=k[X_1, \dots , X_d]$ over a field $k$. 
We consider 
the multi-Rees algebra $\Rees=\Rees(\bf I)$ of the ideals $I_1, \dots , I_r$ and its normalization $\bar \Rees=\bar \Rees(\bf I)$.  
Then the ring extension $\Rees \subset \bar \Rees$ is module-finite because $\Rees$ is a finitely generated algebra over a field and an integral 
domain. Also, $\bar \Rees$ is a normal semigroup ring so that it is Cohen-Macaulay by Hochster's Theorem. 
Thus, by localizing at the homogeneous maximal ideal 
$(X_1, \dots , X_d)$ in $S$, we have the following as a direct consequence of Theorem \ref{multiRees}. 

\setcounter{section}{5}
\setcounter{Theorem}{7}
\begin{Corollary}\label{application}
Let $S=k[X_1, \dots , X_d]$ be a polynomial ring over a field $k$ and 
let $I_1, \dots , I_r$ be arbitrary monomial ideals in $S$. Suppose that 
$$\bf I^{\bf n} \ \mbox{is integrally closed for any} \ \bf n \in \mathbb N^r \ \mbox{with} \ 0 \leq |\bf n| \leq \lambda(I_1 \cdots I_r)-1. $$ 
Then all the power products $\bf I^{\bf n}$ of $I_1, \dots , I_r$ are integrally closed. 
\end{Corollary}

This can be viewed as a common generalization of the original result of  Reid, Roberts and Vitulli  \cite{RRV} and the results of Singla \cite{S}, Sarkar and Verma
\cite{SV}. 
Indeed, if we take $r=1$, then we can readily get the results in \cite{RRV, S}. 
The result \cite{SV} is a special case of Corollary \ref{application} 
where $I_1, \dots , I_r$ are $(X_1, \dots , X_d)$-primary monomial ideals.

Let me explain the construction of this article. 
Sections 2 and 3 are preliminaries. 
In section 2, 
we will recall the notion of filter-regular sequences of multigraded modules and its basic properties, 
in particular, a connection with the vanishing of homogeneous components of the local cohomology modules. 
In section 3, we will recall some notion associated to multigraded modules; in particular, 
the spread, complete and joint reductions. 
We will prove the key result, Theorem \ref{compred}, on the existence of certain complete and joint reductions. 
This can be viewed as a common generalization 
to both of the results \cite{KR} about the existence of complete and joint reductions and 
\cite{Tr} about the existence of filter-regular sequences. 
Our proof yields considerable simplifications of the original one in \cite{KR, R}. 
In section 4, we will prove Theorem \ref{main}. In fact, we will prove a more general result, Theorem \ref{general}, and obtain Theorem \ref{main} as 
a direct consequence.  
In section 5, we will give some applications of Theorems \ref{main} and \ref{general}. We will apply our results to multi-Rees algebras of ideals. 
Theorem \ref{multiRees} and Corollary \ref{application} stated above will be proved in this section. 
To do this, we will give some fundamental results on the spread and the $\bf a$-invariant vector of multi-Rees algebras of ideals. 

Throughout this article, $r>0$ is a fixed positive integer. Let $\mathbb N$ be the set of natural numbers with zero. 
For a finite set $\sigma$, ${}^{\sharp} \sigma$ denotes the number of elements of $\sigma$. 
Vectors in $\mathbb Z^r$
will be written by bold-faced letters, e.g., $\bf a, \bf b, \bf n, \bf m, \cdots$. For a vector 
$\bf n=(n_1, \dots , n_r) \in \mathbb Z^r$, 
we set $|\bf n|=n_1+\dots +n_r$ the sum of all entries of $\bf n$. 
We will use the partial order on $\mathbb Z^r$ as follows: 
$$\bf n \geq \bf m \stackrel{\rm def}{\Leftrightarrow}
n_i \geq m_i \ \mbox{for all} \ i=1, \dots , r. $$
Let $\bf 0=(0, \dots , 0)$ be the zero vector and $\bf n \gg \bf 0$ means that 
each entry $n_i$ of $\bf n$ is large enough, i.e., $n_i \gg 0$ for all 
$i=1, \dots , r$.  
Let $\bf e_i=(0, \dots , \overset{i}{\check{1}}, \dots , 0)$ be the $i$th standard base of $\mathbb Z^r$ and set 
$\bf e=\bf e_1+\dots +\bf e_r=(1, \dots , 1)$. For ideals $I_1, \dots , I_r$, we write $\bf I^{\bf n}=I_1^{n_1} \cdots I_r^{n_r}$ for short.
Similarly, we write $\bf t^{\bf n}=t_1^{n_1} \cdots t_r^{n_r}$ for indeterminates $t_1, \dots  ,t_r$.

\setcounter{section}{1}

\section{Filter-regular sequences and local cohomology}

Let $R$ be a Noetherian standard $\mathbb N^r$-graded ring with $R_{\bf 0}=A$ and 
let $M$ be a finitely generated $\mathbb Z^r$-graded $R$-module. 
We will recall the notion of $M$-filter-regular sequences and collect some basic properties. 
All of the results in this section are well-known. 
We will give a proof of some of them for the reader's convenience. 
See \cite{BS, HHR, Hy, SV, Tr} for more details. 

\begin{Definition}
{ \ }
\begin{enumerate}
\item A homogeneous element $a \in R$ is called $M$-filter-regular $($with respect to $R_{++})$, if 
$$a \notin P \ \mbox{for any} \ P \in \Ass_R (M) \setminus {\rm V}(R_{++})$$
where ${\rm V}(R_{++})$ is the set of prime ideals of $R$ which contains $R_{++}$.  
\item A sequence $a_1, \dots , a_s \in R$ of homogeneous elements is called 
an $M$-filter-regular sequence $($with respect to $R_{++})$, if 
$$a_i \  \mbox{is} \ (M/(a_1, \dots , a_{i-1})M)\mbox{-filter-regular for all} \  i=1, \dots , s, $$ 
where we set $(a_1, \dots , a_{i-1})M=(0)$ when $i=1$. 
\end{enumerate}
\end{Definition}

Note that we do not require the set $\Ass_R (M) \setminus {\rm V}(R_{++}) \neq \emptyset$ for a homogeneous element $a \in R$ to be 
$M$-filter-regular. Thus, when
$\Ass_R (M) \setminus {\rm V}(R_{++})=\emptyset$, any homogeneous element of $R$ is $M$-filter-regular by definition. 

$M$-filter-regular sequences behave well as $M$-sequences. 

\begin{Lemma}\label{filreg}
Let $a_1, \dots , a_s \in R$ be an $M$-filter-regular sequence. 
Then for any $i=1, \dots , s$, the equality 
$$[(a_1, \dots , a_{i-1})M:_M a_i]_{\bf n}=[(a_1, \dots , a_{i-1})M]_{\bf n}$$
holds true for all $\bf n \gg \bf 0$.
\end{Lemma}

\begin{proof}
It is enough to show that for an $M$-filter-regular element $a \in R$, 
$$[0:_M a]_{\bf n}=(0) \ \mbox{for all} \ \bf n \gg \bf 0, $$
because 
$$\left[ 0:_{M/(a_1, \dots , a_{i-1})M} a_i \right]_{\bf n} =[(a_1, \dots , a_{i-1})M:_M a_i]_{\bf n}/[(a_1, \dots , a_{i-1})M]_{\bf n}$$ 
for any 
$i=1, \dots , s$. Let $a \in R$ be an $M$-filter-regular element. Then 
$$\Ass_R(0:_M a) \subseteq {\rm V}(R_{++}). $$ 
Indeed, 
the case $\Ass_R(M) \setminus {\rm V}(R_{++})=\emptyset$ is clear. Suppose $\Ass_R(M) \setminus {\rm V}(R_{++}) \neq \emptyset$. 
Then $a \notin P$ for all $P \in \Ass_R (M) \setminus {\rm V}(R_{++})$. 
Hence, $(0:_M a)_P =(0)$ for all $P \in \Ass_R (M) \setminus {\rm V}(R_{++})$. 
Thus, $\Ass_R( 0:_M a) \subseteq {\rm V}(R_{++})$. 
Hence, $\sqrt{\Ann_R(0:_M a)} \supseteq R_{++} $ so that 
$$R_{++}^n [0:_M a]=(0)$$ 
for all $n \gg 0$. 
Since $[0:_M a]$ is finitely generated, we get $[0:_M a]_{\bf n}=(0)$ for all $\bf{n} \gg \bf 0$. 
\end{proof}

\begin{Remark}
{\rm 
\quad
\begin{enumerate}
\item If $a \in R$ is $M$-filter-regular, then $[0:_M a]$ is $R_{++}$-torsion. 
\item The converse of Lemma \ref{filreg} also holds true. That is, 
the equalities 
$$[(a_1, \dots , a_{i-1})M:_M a_i]_{\bf n}=[(a_1, \dots , a_{i-1})M]_{\bf n} \ \mbox{for all} \ \bf n \gg \bf 0$$ 
and any $i=1, \dots , s$ imply that the sequence $a_1, \dots , a_s$ is an $M$-filter-regular sequence. 
\end{enumerate}}
\end{Remark}

Lemma \ref{filreg} implies the following useful exact sequence of multigraded local cohomology modules 
with respect to the irrelevant ideal $R_{++}$. 

\begin{Lemma}\label{filex}
Let $a \in R_{\bf m}$ be an $M$-filter-regular element of degree $\bf m \in \mathbb N^r$. Then for any $i \geq 0$ and any $\bf n \in \mathbb Z^r$, 
there exists the following exact sequence of $A$-modules: 
$$[H^i_{R_{++}}(M)]_{\bf n} \to [H^i_{R_{++}}(M/aM)]_{\bf n} \to [H^{i+1}_{R_{++}}(M)]_{\bf n - \bf m}. $$
\end{Lemma}

\begin{proof}
Let $ a \in R_{\bf m}$ be an $M$-filter regular element. Consider the exact sequence 
$$0 \to [0:_M a] \to M \to M/[0:_M a] \to 0 $$
of graded $R$-modules.  
Since $[0:_M a]$ is $R_{++}$-torsion, $H^i_{R_{++}}(0:_M a)=(0)$ for all $i \geq 1$ by \cite[Corollary 2.1.7]{BS}. 
Hence, 
$$H^i_{R_{++}}(M) \cong H^i_{R_{++}}(M/[0:_M a]) \ \mbox{for all} \ i \geq 1$$
as graded $R$-modules. 
On the other hand,  
the short exact sequence
$$0 \to \left( M/[0:_M a] \right) (-\bf m) \stackrel{\cdot a}{\to} M \to M/aM \to 0 $$
of graded $R$-modules implies the following exact sequence  
$$H^i_{R_{++}}(M) \to H^i_{R_{++}}(M/aM) \to H^{i+1}_{R_{++}}(M/[0:_M a]) (-\bf m) $$
for any $i \geq 0$. By taking degree $\bf n$ part, we get the desired exact sequence. 
\end{proof}

By using the exact sequence in Lemma \ref{filex}, we have the following. 

\begin{Proposition}\label{filvanish}
Let $a_1, \dots , a_s \in R$ be an $M$-filter-regular sequence with $a_i \in R_{\bf m_i}$ for $i=1, \dots , s$. 
Let $\bf b \in \mathbb Z^r$ and assume that 
$$[H^i_{R_{++}}(M)]_{\bf n}=(0)$$ 
for all $i \geq 0$ and all $\bf n \geq \bf b$. Then we have that 
$$[H^i_{R_{++}}(M/(a_1, \dots , a_s)M)]_{\bf n}=(0)$$
for all $i \geq 0$ and all $\bf n \geq \bf b + \bf m_1+\cdots +\bf m_s$. 
\end{Proposition} 

\begin{proof}
We use induction on $s$. 
By Lemma \ref{filex}, for any $i \geq 0$ and any $\bf n \in \mathbb Z^r$, there exists the exact sequence
$$[H^i_{R_{++}}(M)]_{\bf n} \to [H^i_{R_{++}}(M/a_1M)]_{\bf n} \to [H^{i+1}_{R_{++}}(M)]_{\bf n-\bf m_1}. $$
By assumption, 
$$[H^i_{R_{++}}(M)]_{\bf n}=[H^{i+1}_{R_{++}}(M)]_{\bf n-\bf m_1}=(0)$$ 
for all $i \geq 0$ and all $\bf n \geq \bf b +\bf m_1$. Hence, 
$$[H^i_{R_{++}}(M/a_1M)]_{\bf n}=(0)$$
for all $i \geq 0$ and all $\bf n \geq \bf b +\bf m_1$. 
Thus, we get the case $s=1$.  
Suppose $s \geq 2$. Then $[H^i_{R_{++}}(M/a_1M)]_{\bf n}=(0)$ for all $i \geq 0$ and all $\bf n \geq \bf b+\bf m_1$. 
Since $a_2, \dots , a_s$ is an $(M/a_1M)$-filter-regular sequence,   
\begin{eqnarray*}
\left[H^i_{R_{++}} \left( M/(a_1, a_2, \dots , a_s)M \right) \right]_{\bf n} &=& 
\left[H^i_{R_{++}} \left( (M/a_1M)/(a_2, \dots , a_s)(M/a_1M) \right) \right]_{\bf n} \\
&=& (0)
\end{eqnarray*}
for all $i \geq 0$ and all $\bf n \geq (\bf b+\bf m_1)+\bf m_2+\dots +\bf m_s$ by induction hypothesis. 
Thus, we get the desired vanishing. 
\end{proof}

Let $\varphi: \mathbb Z^r \to \mathbb Z^q$ be a group homomorphism satisfying $\varphi (\mathbb N^r) \subseteq \mathbb N^q$. 
We set 
$$R^{\varphi}=\bigoplus_{\bf m \in  \mathbb N^q} \left( \bigoplus_{\varphi( \bf n )=\bf m} R_{\bf n} \right).  $$
For any $\mathbb Z^r$-graded module $L$, we set 
$$L^{\varphi}=\bigoplus_{\bf m \in \mathbb Z^q} \left( \bigoplus_{\varphi( \bf n )=\bf m} L_{\bf n} \right).  $$
Then $R^{\varphi}$ is an $\mathbb N^q$-graded ring and $L^{\varphi}$ is a $\mathbb Z^q$-graded $R^{\varphi}$-module. 
The local cohomology modules are compatible with a change of grading. 

\begin{Lemma}\label{changeofgrading}\cite[Lemma 1.1]{HHR}
Let $\varphi: \mathbb Z^r \to \mathbb Z^q$ be a group homomorphism satisfying $\varphi (\mathbb N^r) \subseteq \mathbb N^q$. 
Then we have that for any $i$, 
$$\left( H_{\fkM}^i(M) \right)^{\varphi} = H_{\fkM^{\varphi}}^i(M^{\varphi}) $$
where $\fkM = \frak m R+R_{+}$ is the homogeneous maximal ideal of $R$. 
\end{Lemma}

We need one more property on the local cohomology modules. 
Let $I, J$ be subsets of $\{1, \dots , r\}$ such that $J \neq \emptyset$ and $I \cap J =\emptyset$. Then we set
$$\fkM_{I, J}:=\left[ \bigcap_{i \in I} R_{\bf e_i}R \right] \cap \left[ \sum_{j \in J} R_{\bf e_j}R \right]. $$
Note that 
$\fkM_{\{1, \dots , r-1\}, \{r\}}=R_{++}$. With this notation, 
Sarkar and Verma proved the following, 
which can be viewed as a multigraded version of \cite[Lemma 2.3]{Hy}, 
by using the method of \cite[Theorem 3.2.6]{TaiHa}.

\begin{Proposition}\label{SV}\cite[Proposition 3.2]{SV}
Let $\bf a=(a_1, \dots , a_r) \in \mathbb Z^r$. Suppose that 
$$[H^i_{\frak M}(M)]_{\bf n}=(0)$$
for all $i \geq 0$ and 
all $\bf n \in \mathbb Z^r$ with $n_k > a_k$ for some $k=1, \dots , r$. Then 
for any subsets $I, J \subseteq \{1, \dots , r\}$ such that $J \neq \emptyset$ and $I \cap J =\emptyset$, 
we have that 
$$[H^i_{\fkM_{I, J}}(M)]_{\bf n}=(0)$$
for all $i \geq 0$ and all $\bf n \geq \bf a + \bf e$. 
\end{Proposition}

This can be proved by induction on the number ${}^{\sharp} I$ of the set $I$ and by using the Mayer-Vietoris sequence of local 
cohomology modules. As is noted above, if we take $I=\{1, \dots , r-1\}$ and $J=\{r\}$, then $\fkM_{\{1, \dots , r-1\}, \{r\}}=R_{++}$. 
Here is the special case of Proposition \ref{SV} which we will use later. 

\begin{Proposition}\label{lcvanish}
Let $\bf a=(a_1, \dots , a_r) \in \mathbb Z^r$. Suppose that 
$$[H^i_{\frak M}(M)]_{\bf n}=(0)$$
for all $i \geq 0$ and 
all $\bf n \in \mathbb Z^r$ with $n_k > a_k$ for some $k=1, \dots , r$. Then we have that 
$$[H^i_{R_{++}}(M)]_{\bf n}=(0)$$
for all $i \geq 0$ and all $\bf n \geq \bf a + \bf e$. 
\end{Proposition}

\section{Complete and joint reductions}

In this section, we will prove the key result in this article which plays an important role in our proof of Theorem \ref{main}. 
Recall that for a finitely generated $\mathbb Z^r$-graded $R$-module $M$, the homogeneous support of $M$ is defined to be the set 
$$\Supp_{++}(M)=\{ P \in \Spec R \mid M_P \neq (0), P \ \mbox{is graded and} \ R_{++} \nsubseteq P \}. $$
In particular, we set $\Proj^r(R):=\Supp_{++}(R). $

We begin with the following which is the special case of \cite[Lemma 2.2]{Ha}. 

\begin{Lemma}\label{homsupport}
The following are equivalent. 
\begin{enumerate}
\item $\Supp_{++}(M)=\emptyset$
\item $M_{\bf n}=(0)$ for all $\bf n \gg \bf 0$
\item $\Supp_{++}(M/\fkm M)=\emptyset$
\end{enumerate}
\end{Lemma}

\begin{proof}
Note that \begin{eqnarray*}
\Supp_{++}(M) = \emptyset &\Leftrightarrow & R_{++} \subseteq \sqrt{\Ann_R(M)} \\
&\Leftrightarrow& R_{++}^{n} \subseteq \Ann_R(M) \ \mbox{for all} \ n \gg 0 \\
&\Leftrightarrow& R_{++}^{n}M=(0) \ \mbox{for all} \ n \gg 0 . 
\end{eqnarray*}
Since $M$ is finitely generated, the last condition is equivalent to 
$$M_{\bf n}=(0) \ \mbox{for all} \ \bf n \gg \bf 0.$$  
Hence, we have (1) $\Leftrightarrow$ (2). Therefore,  
$$\Supp_{++}(M/\fkm M)=\emptyset \Leftrightarrow M_{\bf n}/\fkm M_{\bf n}=(0) \ \mbox{for all} \ 
\bf n \gg \bf 0. $$ This is equivalent to $M_{\bf n}=(0)$ for all $\bf n \gg \bf 0$ by Nakayama's lemma. Hence, we have $(2) \Leftrightarrow (3)$. 
\end{proof}

The spread $s(M)$ of $M$, which was first introduced by Kirby and Rees \cite{KR}, 
is defined to be
one more than the dimension of the homogeneous support of $M/\fkm M$, i.e.,   
$$s(M)=\dim \Supp_{++}(M/\fkm M)+1. $$ 
Here we set $\dim \emptyset =-1$. Thus, 
the spread is an invariant associated to a multigraded module with a non-negative integer 
and, by Lemma \ref{homsupport}, it is zero if and only if its homogeneous support is empty. 
The spread can be described as the Krull dimension of a certain $\mathbb Z$-graded module associated to $M$. 
We set
$$R^{\Delta}=\bigoplus_{n \in \mathbb N} R_{n \bf e}$$
and call it the diagonal subalgebra of $R$. Similarly, we set 
$$M^{\Delta}=\bigoplus_{n \in \mathbb Z} M_{n \bf e}$$ 
and call it the diagonal submodule of $M$. Then 
$R^{\Delta}$ is a Noetherian standard $\mathbb N$-graded ring with $(R^{\Delta})_0=A$ and 
$M^{\Delta}$ is a finitely generated $\mathbb Z$-graded $R^{\Delta}$-module. 
With this notation, we have the following which can be found in the proof of \cite[Lemma 1.7]{KR}. 

\begin{Lemma}\label{spread}
The map 
$$\phi: \Supp_{++}(M) \to \Supp_{++}(M^{\Delta})$$ 
defined by $\phi(P)=P \cap R^{\Delta}$ is well-defined and bijective. 
\end{Lemma}

\begin{proof}
Note that $\Supp_{++}(M)=\emptyset \Leftrightarrow \Supp_{++}(M^{\Delta})=\emptyset$ by Lemma \ref{homsupport}. 
So, the assertion is clear in this case. Suppose that $\Supp_{++}(M)\neq \emptyset$. 
Let $P \in \Supp_{++}(M)$ and set $\fkp:=\phi(P)=P \cap R^{\Delta}$. 

We first show that $\fkp \in \Supp_{++}(M^{\Delta})$. 
It is clear that $\fkp$ is a graded prime ideal of $R^{\Delta}$. 
Suppose that $(M^{\Delta})_{\fkp}=(0)$. 
Then there exists $a \in R^{\Delta} \setminus \fkp$ such that $a M^{\Delta}=(0)$. 
Let $x \in M$ be a homogeneous element. 
Then there exists 
$$b \in R \setminus P \ \mbox{such that} \ bx \in M^{\Delta}$$ 
because $P \nsupseteq R_{++}$ so that 
$P \nsupseteq R_{\bf e_i} \ \mbox{for all} \ i=1, \dots , r. $ 
Therefore, $ab \in R \setminus P$ and $(ab)x=a(bx)=0$. This contradicts to the fact that 
$M_P \neq (0)$. Hence, $(M^{\Delta})_{\fkp} \neq (0)$. 
Suppose that $\fkp \supseteq (R^{\Delta})_{+}$. Then 
$$P \supseteq \fkp R \supseteq (R^{\Delta})_+R=R_{\bf e}R=R_{++}$$ 
which is a contradiction. Hence, $\fkp \nsupseteq (R^{\Delta})_{+}$. 
Thus, $\fkp \in \Supp_{++}(M^{\Delta})$. 

We next show that $\phi$ is injective. Let $P, P' \in \Supp_{++}(M)$ such that 
$P \cap R^{\Delta}=P' \cap R^{\Delta}$. It is enough to show that 
$P \subseteq P'$. Take any $a \in P$. 
Since $P$ is graded, we may assume that $a$ is homogeneous. 
Then there exists 
$$b \in R \setminus P' \ \mbox{such that} \ ab \in R^{\Delta}$$ 
because $P' \nsupseteq R_{++}$ so that  $P' \nsupseteq R_{\bf e_i}$ for all $i=1, \dots , r$. 
Hence, 
$$ab \in P \cap R^{\Delta}=P' \cap R^{\Delta} \subseteq P'. $$ 
Therefore, $a \in P'$, and, hence, $P \subseteq P'$.

We finally show that $\phi$ is surjective. 
Take any $\fkp \in \Supp_{++}(M^{\Delta})$ and set 
$$P:=\left\{ a \in R \mid a=\sum_i a_i \ \mbox{where} \ a_i \ \mbox{is homogeneous and} \ a_i R \cap R^{\Delta} \subseteq \fkp \right\}. $$
Then $P$ is a graded ideal of $R$. If $P \supseteq R_{++}$, then $R_{\bf e} \subseteq \fkp$ so that 
$R_{\bf e}R^{\Delta} =(R^{\Delta})_+ \subseteq \fkp$. This is a contradiction. Hence, $P \nsupseteq R_{++}$. Let 
$a, b \in R$ be homogeneous elements such that $ab \in P$. 
Suppose that 
$$a \notin P \ \mbox{and} \  b \notin P. $$
Then 
$$aR \cap R^{\Delta} \nsubseteq \fkp \ \mbox{and} \ bR \cap R^{\Delta} \nsubseteq \fkp$$ 
so that there exist elements
$$f \in (aR \cap R^{\Delta}) \setminus \fkp \ \mbox{and} \ g \in (bR \cap R^{\Delta}) \setminus \fkp. $$ 
Then  
$$fg \in abR \cap R^{\Delta} \subseteq \fkp $$
since $ab \in P$. This is a contradiction. Therefore, $a \in P$ or $b \in P$, 
and, hence, $P$ is a prime ideal of $R$. It is clear that $M_P \neq (0)$ since $(M^{\Delta})_{\fkp} \neq (0)$.
Thus, $P \in \Supp_{++}(M)$. 
It is also clear that $P \cap R^{\Delta}=\fkp$.  
Consequently, $\phi$ is surjective. 
This completes the proof. 
\end{proof}

\begin{Corollary}\label{spreadformula}\cite[Lemma 1.7]{KR}
The equalities 
$$s(M)=s(M^{\Delta})=\dim (M^{\Delta}/\fkm M^{\Delta})$$ 
hold true.
\end{Corollary} 

\begin{proof}
The first equality is clear because $(M/\fkm M)^{\Delta} =M^{\Delta}/\fkm M^{\Delta}$ and the map 
$$\phi: \Supp_{++}(M/\fkm M) \to \Supp_{++}(M^{\Delta}/\fkm M^{\Delta})$$
defined by $\phi(P)=P \cap (R/\fkm R)^{\Delta}$ is bijective by Lemma \ref{spread}.  
To show the second equality, we may assume that $r=1$. Moreover, 
since $s(M)=s(M/\fkm M)$ and $s(M)=s(R/\Ann_R (M))$, it is enough to show that 
$$s(R)=\dim R$$
for a Noetherian standard $\mathbb N$-graded ring $R$ such that $R_{\bf 0}$ is a field. 
Then 
$\dim R=\height_R R_+=\dim \Proj R +1=s(R). $
\end{proof}

\begin{Remark}
{\rm Let $I_1, \dots , I_r$ be ideals in $A$ and let $R=\Rees(I_1, \dots, I_r)$ be the multi-Rees algebra of the ideals 
$I_1, \dots , I_r$, which is a subalgebra $A[I_1t_1, \dots , I_rt_r]$ of a polynomial ring $A[t_1, \dots , t_r]$. 
Then the diagonal subalgebra of $R$
$$R^{\Delta}=\Rees(I_1 \cdots I_r)$$ 
is the ordinary Rees algebra of the ideal 
$I_1 \cdots I_r$. Therefore, the spread of $R$ 
$$s(R) =\dim (R^{\Delta}/\fkm R^{\Delta})=\lambda(I_1 \cdots I_r)$$ 
is the analytic spread of the ideal $I_1 \cdots I_r$. 
} 
\end{Remark}

Here is the definition of complete and joint reductions of multigraded modules which was introduced by Kirby and Rees \cite{KR}.

\begin{Definition}
Let $\ell \geq 0$ and $\bf q=(q_1, \dots , q_r)  \in \mathbb N^r$. 
\begin{enumerate}
\item A set of homogeneous elements 
$$\{a_{ij} \in R_{\bf e_i} \mid i=1, \dots , r, \ j=1, \dots , \ell \}$$
is called 
a complete reduction of length $\ell$ with respect to $M$, if the equality
$$M_{\bf n}=[(a_{1j}a_{2j} \cdots a_{rj} \mid j=1, \dots , \ell)M]_{\bf n}$$
holds true for all $\bf n \gg \bf 0$. 
\item A set of homogeneous elements 
$$\{ a_{ij} \in R_{\bf e_i} \mid i=1, \dots , r, \ j=1, \dots , q_i \}$$
is called 
a joint reduction of type $\bf q$ with respect to $M$, if the equality 
$$M_{\bf n}=[(a_{ij} \mid i=1, \dots ,  r, \ j=1, \dots , q_i)M]_{\bf n}$$
holds true for all $\bf n \gg \bf 0$. 
\end{enumerate}
We define the empty set to be a complete reduction of length $0$ and a joint reduction of type $\bf 0$ with respect to $M$. 
\end{Definition}

\begin{Remark}\label{joint}
{\rm 
Let $\bf q \in \mathbb N^r$ such that $|\bf q|=\ell$. Then, as Kirby and Rees noted in \cite{KR}, 
joint reductions of type $\bf q$ with respect to $M$ can be constructed from 
a complete reduction of length $\ell$ with respect to $M$. Indeed, let 
$$\mathcal I:=\{ a_{ij} \in R_{\bf e_i} \mid i=1, \dots , r, \ j=1, \dots , \ell\}$$ 
be a complete reduction of length $\ell$ with respect to $M$. 
Let $\sigma_1, \dots , \sigma_r$ be any partition of $\{1, \dots , \ell\}$ 
into $r$-sets such that ${}^{\sharp} \sigma_i=q_i$, i.e., 
$$\sigma_1 \amalg \dots \amalg \sigma_r=\{1, \dots , \ell \}. $$
Let $\sigma_i=\{s_{i1}, \dots , s_{iq_i} \}$. Then the set  
$$\mathcal J:=\{a_{i s_{ij}} \in R_{\bf e_i} \mid i=1, \dots , r, \ j=1, \dots , q_i \} \subset \mathcal I$$ 
is a joint reduction of type $\bf q$ with respect to $M$ because
$$M_{\bf n}=[(a_{1j}a_{2j} \cdots a_{rj} \mid j=1, \dots , \ell)M]_{\bf n} \subseteq [\mathcal J M]_{\bf n} \subseteq M_{\bf n}$$ 
for all $\bf n \gg \bf 0$. }
\end{Remark}

Here is our key result which plays an important role in our proof of Theorem \ref{main}.

\begin{Theorem}\label{compred}
Assume that the residue field $A/\fkm$ is infinite and let $s=s(M)$. Then there exists a complete reduction of length $s$ with respect to $M$ 
$$\{a_{ij} \in R_{\bf e_i} \mid i=1, \dots , r, \ j=1, \dots , s \}$$ 
such that for any $i_1, \dots ,  i_s \in \{1, \dots , r\}$, 
$$a_{i_1 1}, a_{i_2 2}, \dots , a_{i_s s} \ \mbox{is an} \ M\mbox{-filter-regular sequence}. $$

In particular, for any $\bf q \in \mathbb N^r$ such that $|\bf q|=s$, there exists a joint reduction of type $\bf q$ with respect to $M$ 
$$\{a_{ij} \in R_{\bf e_i} \mid i=1, \dots , r, \ j=1, \dots , q_i \}$$ 
such that 
$$a_{11}, a_{12}, \dots , a_{1 q_1}, a_{21}, \dots , a_{2 q_2}, \dots , a_{r q_r} \ \mbox{is an} \ M\mbox{-filter-regular sequence}. $$ 
\end{Theorem}

Before the proof of Theorem \ref{compred}, we need the following lemma. 

\begin{Lemma}\label{key}
Let $0 \leq t < s$ and consider a set 
$$\{ a_{ij} \in R_{\bf e_i} \mid i=1, \dots , r, \ j=1, \dots,  t \}$$ 
such that 
$$b_1, \dots , b_t \ \mbox{is a subsystem of parameters for} \  M^{\Delta}/\fkm M^{\Delta}$$ 
where $b_j:=a_{1j} a_{2j} \cdots a_{rj} \in R_{\bf e}$ for $j=1, \dots , t$. Then for any finite subset 
$\mathcal X \subset \Proj^r(R)$ which allows the empty set, there exists a set 
$$\{ a_{i t+1} \in R_{\bf e_i} \mid i=1, \dots , r \}$$ such that 
\begin{itemize}
\item $a_{i t+1} \notin Q$ for any $Q \in \mathcal X$, 
\item $b_{t+1}:=a_{1 t+1} a_{2 t+1} \cdots a_{r t+1}$ is a parameter for $M^{\Delta}/(\fkm M^{\Delta}+(b_1, \dots , b_t)M^{\Delta}). $
\end{itemize}
\end{Lemma}

\begin{proof}
Let $\Assh_{R^{\Delta}} (M^{\Delta}/(\fkm M^{\Delta}+(b_1, \dots , b_t)M^{\Delta}))=\{ \fkp_1, \dots , \fkp_u \}$. Then 
$(R^{\Delta})_+ \nsubseteq \fkp_i$ for all $i=1, \dots , u$. Indeed, if $(R^{\Delta})_+ \subseteq \fkp_i$ for some $i=1, \dots , u$, then
\begin{eqnarray*}
0&=&\dim( R^{\Delta}/(\fkm R^{\Delta}+(R^{\Delta})_+ ) )\\
&\geq& \dim(R^{\Delta}/\fkp_i) \\
&=&\dim (M^{\Delta}/(\fkm M^{\Delta}+(b_1, \dots , b_t)M^{\Delta})). 
\end{eqnarray*}
This contradicts to 
$$\dim (M^{\Delta}/(\fkm M^{\Delta}+(b_1, \dots , b_t)M^{\Delta}))=\dim (M^{\Delta}/\fkm M^{\Delta}) -t =s-t>0. $$ 
Hence,  
$$\Assh_{R^{\Delta}} (M^{\Delta}/(\fkm M^{\Delta}+(b_1, \dots , b_t)M^{\Delta})) \subseteq \Supp_{++} (M^{\Delta}/(b_1, \dots ,b_t)M^{\Delta}). $$
Note that $M^{\Delta}/(b_1, \dots ,b_t)M^{\Delta}=(M/(b_1, \dots , b_t)M)^{\Delta}$ since $b_1, \dots , b_t \in R_{\bf e}$. 
By Lemma \ref{spread}, for each $\fkp_i$ there exists the unique $P_i \in \Supp_{++} (M/(b_1, \dots , b_t)M)$ such that 
$\fkp_i=P_i \cap R^{\Delta}$. Then each $P_i \nsupseteq R_{\bf e_k}$ for any $k=1, \dots , r$ because $P_i \nsupseteq R_{++}$. 
On the other hand, if we let 
$$\mathcal X=\{Q_1, \dots , Q_v\}$$
which allows the empty set, then each $Q_j \nsupseteq R_{\bf e_k}$ for any $k=1, \dots , r$ because $Q_j \nsupseteq R_{++}$. 
Therefore, we have that for any $k=1, \dots , r$, any $i=1, \dots , u$ and $j=1, \dots , v$, 
\begin{eqnarray*}
P_i \cap R_{\bf e_k} &\subsetneq& R_{\bf e_k} \\
Q_j \cap R_{\bf e_k} &\subsetneq& R_{\bf e_k}.  
\end{eqnarray*}
Therefore, 
\begin{eqnarray*}
U_i &:=& ([P_i \cap R_{\bf e_k}]+\fkm R_{\bf e_k})/\fkm R_{\bf e_k} \subsetneq R_{\bf e_k}/\fkm R_{\bf e_k}, \\
V_j&:=& ([Q_j \cap R_{\bf e_k}]+\fkm R_{\bf e_k})/\fkm R_{\bf e_k} \subsetneq R_{\bf e_k}/\fkm R_{\bf e_k}
\end{eqnarray*}
by Nakayama's Lemma. Since the residue field $A/\fkm$ is infinite, 
$$W:=\left[\bigcup_{i=1}^u U_i \right]
\cup
\left[\bigcup_{j=1}^v V_j \right] \subsetneq R_{\bf e_k}/\fkm R_{\bf e_k}. $$
Hence, there exists $a_{k t+1} \in R_{\bf e_k}$ such that $\bar{a_{k t+1}} \in (R_{\bf e_k}/\fkm R_{\bf e_k}) \setminus W$. 
Thus, 
$$a_{k t+1} \notin \left[\bigcup_{i=1}^u P_i \right] \cup \left[ \bigcup_{j=1}^v Q_j \right]. $$
Then it is clear that $a_{k t+1} \notin Q$ for any $Q \in \mathcal X$ and 
$b_{t+1}:=a_{1 t+1} a_{2 t+1} \cdots a_{r t+1} \in R_{\bf e}$ is a parameter for 
$M^{\Delta}/(\fkm M^{\Delta} + (b_1, \dots , b_t)M^{\Delta})$ because 
$$b_{t+1} \notin \bigcup_{i=1}^u \fkp_i.  $$
This completes the proof. 
\end{proof}

\begin{proof}[Proof of Theorem \ref{compred}]
The assertion is clear when $s=0$. Suppose $s \geq 1$ and let $0 \leq t < s$. 
Assume that there exists a set of homogeneous elements 
$$\{a_{ij} \in R_{\bf e_i} \mid i=1, \dots , r, \ j=1, \dots , t\}$$
such that 
\begin{itemize}
\item $a_{i_1 1}, \dots , a_{i_t t}$ is an $M$-filter-regular sequence for any $i_1, \dots , i_t \in \{1, \dots, r\}$, 
\item $b_1, \dots , b_t$ is a subsystem of parameters for $M^{\Delta}/\fkm M^{\Delta}$
\end{itemize}
where $b_j=a_{1j}a_{2j} \cdots a_{rj} \in R_{\bf e}$ for $j=1, \dots , t$. 
Then by applying Lemma \ref{key} for the set 
$$\mathcal X:=\bigcup_{i_1, \dots , i_t \in \{1, \dots, r\}} \Ass_R(M/(a_{i_1 1}, \dots , a_{i_t t})M) \setminus {\rm V}(R_{++})$$
which allows the empty set, we have a set of homogeneous elements 
$$\{ a_{i t+1} \in R_{\bf e_i} \mid i=1, \dots , r\}$$ 
such that
\begin{itemize}
\item $a_{i_1 1}, \dots , a_{i_t t}, a_{i t+1}$ is an $M$-filter-regular sequence for any $i_1, \dots , i_t, i \in \{1, \dots , r\}$, 
\item $b_1, \dots , b_t, b_{t+1}$ is a subsystem of parameters for $M^{\Delta}/\fkm M^{\Delta}$
\end{itemize}
where $b_{t+1}=a_{1 t+1}a_{2 t+1} \cdots a_{r t+1} \in R_{\bf e}$. 
Thus, by repeating this procedure, we can construct a set of homogeneous elements 
$$\{a_{ij} \in R_{\bf e_i} \mid i=1, \dots , r, \ j=1, \dots , s\}$$ 
such that
\begin{itemize}
\item $a_{i_1 1},  \dots , a_{i_s s}$ is an $M$-filter-regular sequence for any $i_1, \dots , i_s \in \{1, \dots, r\}$, 
\item $b_1, \dots , b_s$ is a system of parameters for $M^{\Delta}/\fkm M^{\Delta}$
\end{itemize}
where $b_j=a_{1j}a_{2j} \cdots a_{rj} \in R_{\bf e}$ for $j=1, \dots , s$. 
Since $b_1, \dots , b_s$ is a system of parameters for $M^{\Delta}/\fkm M^{\Delta}$, we have that 
\begin{eqnarray*}
0&=&\dim (M^{\Delta}/(\fkm M^{\Delta}+(b_1, \dots , b_s)M^{\Delta})) \\
&=&s(M^{\Delta}/(b_1, \dots , b_s)M^{\Delta}) \\
&=&s(M/(b_1, \dots , b_s)M). 
\end{eqnarray*}
Hence, $\Supp_{++} (M/(b_1, \dots , b_s)M)=\emptyset$. 
By Lemma \ref{homsupport}, $\left[M/(b_1, \dots , b_s)M \right]_{\bf n}=(0)$ for all $\bf n \gg \bf 0$ 
so that 
$$M_{\bf n}=[(b_1, \dots , b_s)M]_{\bf n} \ \mbox{for all} \  \bf  n \gg \bf 0.$$ 
Therefore, the set 
$$\{a_{ij} \in R_{\bf e_i} \mid i=1, \dots , r, \ j=1, \dots , s\}$$ 
is a complete reduction of length $s$ with respect to $M$. 
Then the last assertion on the existence of a joint reduction follows immediately from Remark \ref{joint}. 
\end{proof}

\begin{Remark}
{\rm Suppose that the residue field $A/\fkm$ is infinite and let $\ell \geq s(M)$.  
\begin{enumerate}
\item One can construct a complete reduction of length $\ell$ with respect to $M$ 
with the same property in Theorem \ref{compred}. 
\item Kirby and Rees proved the existence of a complete reduction of length $\ell$ with respect to $M$ $($\cite[Theorem 1.6]{KR}$)$. 
Theorem \ref{compred} improves 
their result. Moreover, our proof yields considerable simplification of the one in \cite{KR} and also the original one in \cite{R}.
\item Sarkar and Verma proved the last assertion in Theorem \ref{compred} under the assumption 
that $A$ is an Artinian local ring and $M$ has positive mixed multiplicities and $\dim M^{\Delta} \geq 1$ (\cite[Theorem 2.3]{SV}). 
Theorem \ref{compred} generalizes and improves their result. 
\item Let $\bf q \in \mathbb N^r$ with $|\bf q|=s(M)$. 
Then Trung proved the existence of $M$-filter-regular sequence of type $\bf q$ in bigraded cases (\cite[Lemma 2.3]{Tr}). 
Thus, Theorem \ref{compred} can be viewed
as a common generalization of \cite{KR, Tr}. 
\end{enumerate}
}
\end{Remark}

\section{Proof of Theorem \ref{main}}

In this section, we will give a proof of Theorem \ref{main}. In fact, we prove the following general result. 

\begin{Theorem}\label{general}
Let $R$ be a Noetherian standard $\mathbb N^r$-graded ring such that $R_{\bf 0}=A$ is a local ring. 
Let $M$ be a finitely generated $\mathbb Z^r$-graded $R$-module and let $N$ be a graded $R$-submodule of $M$. 
Let $\bf b \in \mathbb Z^r$ and $\ell \geq |\bf b|+s(M)-1$. Assume that 
\begin{itemize}
\item $[H^i_{R_{++}}(M)]_{\bf n}=(0)$ for all $i \geq 0$ and all $\bf n \geq \bf b. $
\item $[M/N]_{\bf n}=(0)$ for all $\bf n \geq \bf b$ with $|\bf n| = \ell. $ 
\end{itemize}
Then we have that  
$$[M/N]_{\bf n}=(0) \ \mbox{for all} \ \bf n \geq \bf b \ \mbox{with} \ |\bf n| \geq \ell. $$
\end{Theorem}

\begin{proof}
By passing to the faithfully flat $R$-algebra $R':=R \otimes_A A'$, where $A':=A[T]_{\fkm A[T]}$ is the localization of 
a polynomial ring $A[T]$ at the prime ideal $\fkm A[T]$, 
we may assume that 
the residue field of $R_{\bf 0}=A$ is infinite. 
We set $s=s(M)$. Let $\bf n \geq \bf b$ with $|\bf n| \geq \ell$. We use induction on $|\bf n|$. The case $|\bf n|=\ell$ is clear. 
Suppose $|\bf n| \geq \ell+1$ and assume that 
$$[M/N]_{\bf m}=(0) \ \mbox{for all} \ \bf m \geq \bf b \ \mbox{with} \ |\bf m|=|\bf n|-1.$$
Note that $\bf n-\bf b \geq \bf 0$ and $|\bf n-\bf b| \geq s$ since $\bf n \geq \bf b$ and $|\bf n|\geq \ell+1 \geq |\bf b|+s$.  
Then one can find $\bf q \in \mathbb N^r$ such that $\bf n-\bf b \geq \bf q$ and $|\bf q|=s$. 
By Theorem \ref{compred}, there exists a joint reduction of type $\bf q$ with respect to $M$
$$\mathcal J:=\{a_{ij} \in R_{\bf e_i} \mid i=1, \dots , r, \ j=1, \dots , q_i \}$$
such that 
$$a_{11}, a_{12}, \dots , a_{1 q_1}, a_{21}, \dots , a_{2q_2}, \dots , a_{rq_r} \ \mbox{is an} \ M\mbox{-filter-regular sequence}.$$ 
By Proposition \ref{filvanish}, the assumption
$$[H^i_{R_{++}}(M)]_{\bf \ell}=(0) \ \mbox{for all} \ i \geq 0\ \mbox{and all} \ \bf \ell \geq \bf b $$
implies that  
\begin{eqnarray}\label{vanish}
[H^i_{R_{++}}(M/(a_{11}, \dots , a_{r q_r})M)]_{\bf \ell}=(0) \ \mbox{for all} \ i \geq 0 \ \mbox{and all} \ \bf \ell \geq \bf b + \bf q. 
\end{eqnarray}
Note that, since the set $\mathcal J$ is a joint reduction with respect to $M$, 
$M/(a_{11}, \dots , a_{rq_r})M$ is $R_{++}$-torsion.  
Hence, by (\ref{vanish}), we have that 
$$\left[ M/(a_{11}, \dots , a_{rq_r})M \right]_{\bf \ell}=\left[ H^0_{R_{++}}(M/(a_{11}, \dots , a_{rq_r})M) \right]_{\bf \ell}=(0)$$
for all $\bf \ell \geq \bf b +\bf q$. 
Therefore, since the vector $\bf n$ satisfies $\bf n \geq \bf b + \bf q$, 
\begin{eqnarray*}
M_{\bf n}&=&[(a_{11}, \dots , a_{rq_r})M]_{\bf n} \\
&=&\sum_{i=1}^r \left(\sum_{j=1}^{q_i} a_{ij}M_{\bf n-\bf e_i} \right). 
\end{eqnarray*}
Here, we note that if $s=0$, then $M_{\bf n}=(0)$ so that the assertion is clear. Suppose $s>0$.

\

\noindent
$\bf{\rm Claim}$ For any $i=1, \dots , r$, the equality 
$$\sum_{j=1}^{q_i}a_{ij}M_{\bf n-\bf e_i}=\sum_{j=1}^{q_i}a_{ij}N_{\bf n-\bf e_i}$$
holds true. 

\

Fix any $i=1, \dots , r$ and set $\bf m:=\bf n-\bf e_i$. If $\bf m \ngeq \bf b$, then $m_i=n_i -1 < b_i \leq n_i$. 
Hence, $n_i=b_i$ so that $q_i=0$. 
Thus, the desired equality is clear. Suppose $\bf m \geq \bf b$. Then $|\bf m|=|\bf n-\bf e_i|=|\bf n|-1$. 
By our induction hypothesis, $M_{\bf m}=N_{\bf m}$ so that the desired equality holds true. 

\

Consequently, we have that 
\begin{eqnarray*}
M_{\bf n}&=&\sum_{i=1}^r \left( \sum_{j=1}^{q_i} a_{ij} M_{\bf n-\bf e_i} \right) \\
&=&\sum_{i=1}^r \left( \sum_{j=1}^{q_i} a_{ij} N_{\bf n-\bf e_i} \right) \\
& \subseteq  &N_{\bf n}. 
\end{eqnarray*}
Hence, $M_{\bf n}=N_{\bf n}$. This completes the proof. 
\end{proof}

Let me give a proof of Theorem \ref{main}. 

\begin{proof}[Proof of Theorem \ref{main}.]
We set $s=s(M)$ and let $\bf a(M)=(a^1(M), \dots , a^r(M)) \in \mathbb Z^r$ be the $\bf a$-invariant vector of $M$. 
Suppose that $M$ is a Cohen-Macaulay graded $R$-module. 
Let $\ell \geq |\bf a(M)|+s+r-1$ and assume that 
$$[M/N]_{\bf n}=(0) \ \mbox{for all} \ \bf n \geq \bf a(M) + \bf e \ \mbox{with} \ |\bf n| = \ell. $$ 
Since $M$ is Cohen-Macaulay, $H^i_{\fkM}(M)=(0)$ for all $i \neq \dim_R M$. 
Hence,  
$$[H^i_{\frak M}(M)]_{\bf n}=(0)$$ 
for all $i \geq 0$ and 
all $\bf n \in \mathbb Z^r$ with $\ n_k > a^k(M)$ for some $k=1, \dots , r$. 
Then, by Proposition \ref{lcvanish}, we have that 
$$[H^i_{R_{++}}(M)]_{\bf n}=(0)$$
for all $i \geq 0$ and all $\bf n \geq \bf a(M) + \bf e$.
Note that $|\bf a(M)|+s+r-1=|\bf a(M)+\bf e|+s-1$. 
By taking $\bf b=\bf a(M)+\bf e$ and applying Theorem \ref{general}, we get that 
$$[M/N]_{\bf n}=(0) \ \mbox{for all} \ \bf n \geq \bf a(M) + \bf e \ \mbox{with} \ |\bf n| \geq \ell. $$
This completes the proof. 
\end{proof}

\begin{Remark}
{\rm 
Suppose $s(M)=0$. Then $M$ is $R_{++}$-torsion by the proof of Lemma \ref{homsupport} so that  
$H^0_{R_{++}}(M)=M$. Hence, the assumption 
$$[H^i_{R_{++}}(M)]_{\bf n}=(0) \ \mbox{for all} \ i \geq 0 \ \mbox{and all} \ \bf n \geq \bf b  $$
in Theorem \ref{general} implies that 
$M_{\bf n}=(0)$ for all $\bf n \geq \bf b$. 
In particular, if $M$ is Cohen-Macaulay with $s(M)=0$, then
$$[H^i_{R_{++}}(M)]_{\bf n}=(0) \ \mbox{for all} \ i \geq 0 \ \mbox{and all} \ \bf n \geq \bf a(M)+\bf e  $$
by Proposition \ref{lcvanish}, and, hence, $M_{\bf n}=(0)$ for all $\bf n \geq \bf a(M)+\bf e$. 
}
\end{Remark}

\section{Applications}

In this section, we will apply Theorems \ref{main} and \ref{general} to multi-Rees algebras of filtrations of ideals. 
Let $I_1, \dots , I_r$ be ideals in a Noetherian local ring $A$. 
Let $\Rees(\bf I) =\Rees(I_1, \dots , I_r)$ be the multi-Rees algebra of $I_1, \dots , I_r$, 
which is a subalgebra 
$$\Rees(\bf I)=A[I_1t_1, \dots, I_rt_r]=\sum_{\bf n \in \mathbb N^r} \bf I^{\bf n} \bf t^{\bf n}$$ 
of a polynomial ring $A[t_1, \dots , t_r]$. Let 
$$\mathcal F=\{ F(\bf n)\}_{\bf n \in \mathbb N^r}$$ 
be a filtration of ideals in $A$, 
i.e., each $F(\bf n)$ is an ideal in $A$ such that 
\begin{itemize}
\item $F(\bf n) \supseteq F(\bf m)$ if $\bf n \leq \bf m$
\item $F(\bf n) F(\bf m) \subseteq F(\bf n+\bf m)$. 
\end{itemize}
A filtration $\mathcal F$ is said to be $\bf I$-filtration if the following one more condition is satisfied. 
\begin{itemize}
\item $\bf I^{\bf n} \subseteq F(\bf n)$ for all $\bf n \in \mathbb N^r. $
\end{itemize}
Typical examples of $\bf I$-filtrations are an $\bf I$-adic filtration $\{ \bf I^{\bf n}\}_{\bf n \in \mathbb N^r}$ and 
its integral closure filtration $\{ \bar{\bf I^{\bf n}} \}_{\bf n \in \mathbb N^r}$. 
For an $\bf I$-filtration $\mathcal F$, we set 
$$\Rees(\mathcal F)=\sum_{\bf n \in \mathbb N^r} F(\bf n) \bf t^{\bf n}$$ 
and call it 
the multi-Rees algebra of the filtration $\mathcal F$. $\Rees(\mathcal F)$ is a subalgebra of $A[t_1, \dots , t_r]$ which 
contains the ordinary multi-Rees algebra $\Rees(\bf I)$ as an $A$-subalgebra. 
The multi-Rees algebra of the $\bf I$-adic filtration $\{ \bf I^{\bf n}\}_{\bf n \in \mathbb N^r}$ 
will be denoted by $\Rees (\bf I)$ for short. Also, 
the multi-Rees algebra of the integral closure filtration $\{ \bar{\bf I^{\bf n}} \}_{\bf n \in \mathbb N^r}$ will be denoted by 
$$\bar \Rees(\bf I)=\sum_{\bf n \in \mathbb N^r} \bar{\bf I^{\bf n}} \bf t^{\bf n}, $$
which coincides with the integral closure of $\Rees(\bf I)$ in $A[t_1, \dots , t_r]$
(\cite[Proposition 5.2.1]{HSw}). 

Then our main application can be stated as follows. 

\begin{Theorem}\label{mainapp}
Let $(A, \fkm) $ be a Noetherian local ring of $\dim A=d>0$. 
Let $I_1, \dots , I_r$ be ideals in $A$ and let $\mathcal F$ be an $\bf I$-filtration such that $\dim \Rees(\mathcal F)=d+r$. 
Suppose that 
\begin{itemize}
\item $\Rees(\bf I) \subseteq \Rees(\mathcal F)$ is module-finite, 
\item $\Rees(\mathcal F)$ is Cohen-Macaulay.
\end{itemize} 
Let $\ell \geq \lambda(I_1 \cdots I_r)-1$ and assume that 
$$\bf I^{\bf n}=F(\bf n) \ \mbox{for any} \ \bf n \in \mathbb N^r \ \mbox{with} \ |\bf n| = \ell. $$ 
Then we have that 
$$\bf I^{\bf n}=F(\bf n) \ \mbox{for any} \ \bf n \in \mathbb N^r \ \mbox{with} \ |\bf n| \geq \ell. $$ 

In particular, if 
$$\bf I^{\bf n}=F(\bf n) \ \mbox{for any} \ \bf n \in \mathbb N^r \ \mbox{with} \ 0 \leq |\bf n| \leq \lambda(I_1 \cdots I_r)-1, $$
then 
$\bf I^{\bf n}=F(\bf n)$ for any $\bf n \in \mathbb N^r$, i.e., $\Rees(\bf I)=\Rees(\mathcal F)$. 
\end{Theorem}

In order to prove Theorem \ref{mainapp}, we begin with the following. 

\begin{Proposition}
Let $R \subseteq S$ be a Noetherian $\mathbb N^r$-graded ring extension such that 
\begin{itemize}
\item $R_{\bf 0}=S_{\bf 0}=A$ is a local ring, 
\item $R$ is a standard $\mathbb N^r$-graded ring, and 
\item $R \subseteq S$ is module-finite. 
\end{itemize}
Then $s(R)=s(S)$. 
\end{Proposition}

\begin{proof}

By Corollary \ref{spreadformula}, we may assume that $r=1$ and it is enough to show that 
$$\dim R/\fkm R =\dim S/\fkm S.$$ 
By passing to the ring $A[T]_{\fkm A[T]}$ where $T$ is an indeterminate, we can assume that 
the residue field $A/\fkm $ is infinite. 
Let $\ell:=s(S)=\dim S/\fkm S$. Since $S$ is a finitely generated $R$-module, 
there exists $a_1, \dots , a_{\ell} \in R_1$ such that $a_1, \dots , a_{\ell}$ is a complete reduction of $S$, that is, 
$$S_{n}=[(a_1, \dots , a_{\ell})S]_n \ \mbox{for all} \ n \gg 0.$$ 
Let $T:=A[a_1, \dots , a_{\ell}]$ be the subalgebra of $R$. 
Then $T \subseteq S$ is module-finite, and, hence, the natural map $T/\fkm T \to S/\fkm S$ is also finite. 
Therefore, $\dim T/\fkm T \geq \dim S/\fkm S=\ell$. Consider the exact sequence
$$0 \to K \to (A/\fkm )[X_1, \dots , X_{\ell}] \to T/\fkm T \to 0, $$
where $(A/\fkm )[X_1, \dots , X_{\ell}]$ is a polynomial ring. Then, by comparing the Krull dimensions, we have $K=(0)$, and, hence, 
$(A/\fkm )[X_1, \dots , X_{\ell}] \cong T/\fkm T$. 
Thus, the natural map $T/\fkm T \to S/\fkm S$ is injective. Therefore, the natural map 
$T/\fkm T \to R/\fkm R$ is also injective and finite. Hence,  
$\dim R/\fkm R =\dim T/\fkm T=\ell =\dim S/\fkm S$. 
\end{proof}

As a corollary, we immediately get the following. 

\begin{Corollary}\label{analyticspread}
Let $I_1, \dots , I_r$ be ideals in $A$ and 
let $\mathcal F$ be an $\bf I$-filtration of ideals in $A$ such that 
$\Rees(\bf I) \subseteq \Rees(\mathcal F)$ is module-finite. 
Then we have $s(\Rees(\mathcal F))=\lambda(I_1 \cdots I_r)$.
\end{Corollary}

Next we recall the following general fact on the $a$-invariant. 

\begin{Lemma}\label{a-invariant}
Let $B$ be a Noetherian $\mathbb N^r$-graded ring such that $B_{\bf 0}=A$ is a local ring. 
Let $\mathcal H=\{ H(m)\}_{m \in \mathbb N}$ be
a filtration of graded ideals in $B$. Let $\fkn:=\fkm B+B_+$ be the homogeneous maximal ideal of $B$ and set 
$$\Rees_B(\mathcal H)=\sum_{m \in \mathbb N} H(m)t^m$$ 
the Rees algebra of $\mathcal H$. 
Let $\fkM:=\fkn \Rees_B(\mathcal H)+\Rees_B(\mathcal H)_+$ be the homogeneous maximal ideal of $\Rees_B(\mathcal H)$. 
Consider $\Rees_B(\mathcal H)$ as an $\mathbb N$-graded $B$-algebra and set 
$$a(\Rees_B(\mathcal H))=\sup \{ k \in \mathbb Z \mid [H^{\dim \Rees_B(\mathcal H)}_{\fkM} (\Rees_B(\mathcal H))]_k \neq (0) \}. $$
Assume that $\Rees_B(\mathcal H)$ is Noetherian and $\dim \Rees_B(\mathcal H)=\dim B+1$. 
Then we have 
$$a(\Rees_B(\mathcal H))=-1. $$ 
\end{Lemma}

\begin{proof}
By passing to the ring $B_{\fkn}$, we may assume that $(B, \fkn)$ is a local ring. Then the assertion is well-known 
(see \cite{GN, HHR} for instance).  
\end{proof}

Let $(A, \fkm)$ be a Noetherian local ring of $\dim A=d>0$. 
Let 
$$\mathcal F=\{F(\bf n, m)\}_{(\bf n, m) \in \mathbb N^{r+1}}$$ 
be a filtration of ideals in $A$. We set 
$$\mathcal G=\{ G(\bf n) \}_{\bf n \in \mathbb N^r}:=\{ F(\bf n, 0)\}_{\bf n \in \mathbb N^r}.$$ 
Consider the multi-Rees algebra of $\mathcal F$ and $\mathcal G$ 
$$\Rees(\mathcal F)=\sum_{(\bf n, m) \in \mathbb N^{r+1}} F(\bf n, m) \bf t^{\bf n} t^m \subseteq A[t_1, \dots , t_r, t],  $$
$$\Rees(\mathcal G)=\sum_{\bf n \in \mathbb N^r} G(\bf n)\bf t^{\bf n}=\sum_{\bf n \in \mathbb N^r} F(\bf n, 0)\bf t^{\bf n}
\subseteq A[t_1, \dots , t_r]. $$
Let $\varphi: \mathbb Z^{r+1} \to \mathbb Z$ be a group homomorphism defined by $\varphi(\bf n, m)=m$. 
Then 
$$\Rees(\mathcal F)^{\varphi}=\sum_{m \in \mathbb N} \left( \sum_{\bf n \in \mathbb N^r} F(\bf n, m) \bf t^{\bf n} \right) t^m$$
is the ring $\Rees(\mathcal F)$ as an $\mathbb N$-graded ring with the homogeneous component of degree $m$  
$$H(m):=\sum_{\bf n \in \mathbb N^r} F(\bf n, m) \bf t^{\bf n}. $$
We set 
$$\mathcal H:=\{ H(m) \}_{m \in \mathbb N}.$$ 
Then 
$\mathcal H$ is a filtration of graded ideals in $B:=\Rees(\mathcal G)$ and 
$$\Rees(\mathcal F)^{\varphi}=\Rees_B(\mathcal H)$$ 
as 
$\mathbb N$-graded $B$-algebras. 
Let $\fkM=\fkm \Rees(\mathcal F)+\Rees(\mathcal F)_+$ be 
the homogeneous maximal ideal of $\Rees(\mathcal F)$ as an $\mathbb N^{r+1}$-graded ring. 
Assume that $\Rees_B(\mathcal H)$ is Noetherian and $\dim \Rees_B(\mathcal H)=\dim B+1$. Then, by Lemma \ref{a-invariant}, 
$$a(\Rees(\mathcal F)^{\varphi})=a(\Rees_B(\mathcal H))=-1.$$ 
On the other hand, by Lemma \ref{changeofgrading}, 
$$\left( H_{\fkM}^{\dim \Rees(\mathcal F)}(\Rees(\mathcal F)) \right)^{\varphi}=
H_{\fkM^{\varphi}}^{\dim \Rees(\mathcal F)}(\Rees(\mathcal F)^{\varphi}) $$
as $\mathbb Z$-graded modules. 
Hence, for any $k \in \mathbb Z$, we have 
\begin{eqnarray*}
\left[ H^{\dim \Rees(\mathcal F)}_{\fkM^{\varphi}}(\Rees(\mathcal F)^{\varphi}) \right]_k
&=&\left[ \left( H_{\fkM}^{\dim \Rees(\mathcal F)}(\Rees(\mathcal F))\right)^{\varphi} \right]_k \\
&=&\bigoplus_{\bf n \in \mathbb Z^r} \left[ H^{\dim \Rees(\mathcal F)}_{\fkM}(\Rees(\mathcal F)) \right]_{(\bf n, k)}. 
\end{eqnarray*}
If we set 
$$a^{r+1}(\Rees(\mathcal F)):=\sup \left\{ k \in \mathbb Z \mid \left[ H^{\dim \Rees(\mathcal F)}_{\fkM}(\Rees(\mathcal F)) \right]_{(\bf n, k)} \neq (0) 
\ \mbox{for some} \ \bf n \in \mathbb Z^r \right\}, $$
then 
$$a^{r+1}(\Rees(\mathcal F))=a(\Rees(\mathcal F)^{\varphi})=-1. $$

\smallskip

This observation implies the following. 

\begin{Corollary}\label{a-vector}
Let $(A, \fkm)$ be a Noetherian local ring of $\dim A =d>0$. 
Let $\mathcal F=\{F(\bf n)\}_{\bf n \in \mathbb N^r}$ be a filtration of ideals in $A$. 
Suppose that 
\begin{itemize}
\item $\Rees(\mathcal F)=\sum_{\bf n \in \mathbb N^r} F(\bf n) \bf t^{\bf n}$ is Noetherian
\item $\dim \Rees(\mathcal F)=d+r$. 
\end{itemize}
Then we have 
$$\bf a(\Rees(\mathcal F))=-\bf e=(-1, -1, \dots , -1). $$ 
\end{Corollary}

\begin{proof}
When $r=1$, it is well-known (\cite{GN}). Suppose $r \geq 2$. 
Fix $i=1, \dots , r$ and let 
$B=\Rees(\mathcal G)$ be the $\mathbb N^{r-1}$-graded Rees algebra defined by the filtration 
$$\mathcal G:=\{ G(\bf n) \}_{\bf n \in \mathbb N^{r-1}}$$ 
where $G(\bf n)=F(n_1, \dots, \overset{i}{\check{0}} , \dots , n_r)$. 
Let $\varphi_i: \mathbb Z^r \to \mathbb Z$ be a group homomorphism defined by 
$\varphi( \bf n )=n_i$. Then $\Rees(\mathcal F)^{\varphi_i}$ is the Rees algebra of the filtration of ideals in $B$ and 
$\dim B=d+r-1$. Therefore, $a^i(\Rees(\mathcal F))=a(\Rees(\mathcal F)^{\varphi_i})=-1$ so that $\bf a(\Rees(\mathcal F))=-\bf e$. 
\end{proof}

Theorem \ref{mainapp} is now a direct consequence of Theorem \ref{main}. 

\begin{proof}[Proof of Theorem \ref{mainapp}]
Let $R=\Rees(\bf I)$ and let $M=\Rees(\mathcal F)$. By the assumptions,  
$M$ is a finitely generated Cohen-Macaulay graded $R$-module with $\dim M=d+r$. 
Then $s(M)=\lambda(I_1 \cdots I_r)$ by Corollary \ref{analyticspread} and 
$\bf a (M)=-\bf e$ by Corollary \ref{a-vector}. 
Hence, $| \bf a(M) |+s(M)+r-1=\lambda(I_1 \cdots I_r)-1$. Therefore, we get the assertion by Theorem \ref{general}. 
\end{proof}

The following is the special case of Theorem \ref{mainapp}.

\begin{Theorem}\label{multiRees}
Let $(A, \fkm)$ be a Noetherian local ring of $\dim A=d>0$. 
Let $I_1, \dots , I_r$ be ideals in $A$ such that $\dim \Rees(\bf I)=d+r$. Suppose that
$\Rees(\bf I) \subseteq \bar \Rees(\bf I)$ is module-finite and $\bar \Rees(\bf I)$ is Cohen-Macaulay. 
Let $\ell \geq \lambda(I_1 \cdots I_r)-1$ and assume that 
$$\bf I^{\bf n} \ \mbox{is integrally closed for all} \ \bf n \in \mathbb N^r \ \mbox{with} \ |\bf n| = \ell. $$ 
Then we have that 
$$\bf I^{\bf n} \ \mbox{is integrally closed for all} \ \bf n \in \mathbb N^r \ \mbox{with} \ |\bf n| \geq \ell. $$ 

In particular, if 
$$\bf I^{\bf n} \ \mbox{is integrally closed for any} \ \bf n \in \mathbb N^r \ \mbox{with} \ 0 \leq |\bf n| \leq \lambda(I_1 \cdots I_r)-1, $$
then 
all the power products $\bf I^{\bf n}$ of $I_1, \dots , I_r$ are integrally closed.   
\end{Theorem}

In Theorem \ref{multiRees}, if we assume that 
$A$ is an analytically unramified Noetherian local ring, then the multi-Rees algebra $\Rees(\bf I)$ 
has finite integral closure (\cite{R1961}). 
As a direct consequence of Theorem \ref{multiRees}, we have the following. 

\begin{Corollary}\label{unramified}
Let $A$ be 
an analytically unramified Noetherian local ring of $\dim A=d>0$. 
Let $I_1, \dots , I_r$ be ideals in $A$ such that $\dim \Rees(\bf I)=d+r$. 
Assume that the integral closure $\bar \Rees(\bf I)$ of $\Rees(\bf I)$ is Cohen-Macaulay. 
Suppose that 
$$\bf I^{\bf n} \ \mbox{is integrally closed for any} \ \bf n \in \mathbb N^r \ \mbox{with} \ 0 \leq |\bf n| \leq \lambda(I_1 \cdots I_r)-1. $$ 
Then $\bf I^{\bf n}$ is integrally closed for all $\bf n \in \mathbb N^r$. 
\end{Corollary}

Let $I_1, \dots , I_r$ be monomial ideals in a polynomial ring $S=k[X_1, \dots , X_d]$ over a field $k$. Then,  
the multi-Rees algebra $\Rees(\bf I)$ has finite integral closure since $\Rees(\bf I)$ is a finitely generated algebra over a field and an integral domain. 
Moreover, the integral closure $\bar \Rees(\bf I)$ of $\Rees(\bf I)$ is a normal semigroup ring so that it is Cohen-Macaulay by Hochster's Theorem 
(\cite{Ho}). Thus, by localizing at $(X_1, \dots , X_d)$, we have the following.

\begin{Corollary}\label{application}
Let $S=k[X_1, \dots , X_d]$ be a polynomial ring over a field $k$ and 
let $I_1, \dots , I_r$ be arbitrary monomial ideals in $S$. Suppose that 
$$\bf I^{\bf n} \ \mbox{is integrally closed for any} \ \bf n \in \mathbb N^r \ \mbox{with} \ 0 \leq |\bf n| \leq \lambda(I_1 \cdots I_r)-1. $$ 
Then all the power products $\bf I^{\bf n}$ of $I_1, \dots , I_r$ are integrally closed. 
\end{Corollary}

It would be interesting to know 
whether or not the assumption in Corollary \ref{unramified} that the normalization of the multi-Rees algebra is Cohen-Macaulay is needed if $A$ is a regular local ring. 
In dimension two, it is known as Zariski's theorem on the product of integrally closed ideals
that we do not need the assumption. 
How about in higher dimensions?  
In particular, it would be interesting to know 
whether the following Question holds true or not. 

\begin{Question}
Let $A$ be a ring satisfying that either 
\begin{itemize}
\item $A$ is a polynomial ring over a field, or 
\item $A$ is a regular local ring.
\end{itemize} 
Let $I_1, \dots  ,I_r$ be ideals in $A$. Suppose that 
$$\bf I^{\bf n} \ \mbox{is integrally closed for any} \ \bf n \in \mathbb N^r \ \mbox{with} \ 0 \leq |\bf n| \leq \lambda(I_1 \cdots I_r)-1. $$ 
Then are all the power products $\bf I^{\bf n}$ integrally closed? 
\end{Question}

\section*{Acknowledgments}
The author would like to thank Professor Kazuhiko Kurano for his valuable comments. 
He would also like to thank Professors Shiro Goto, Koji Nishida and Shin-ichiro Iai for their helpful conversations. He thanks the referee for his/her careful reading 
and kindly suggestions.




\begin{thebibliography}{99}

\bibitem{BS}
M. P. Brodmann and R. Y. Sharp, 
Local cohomology: An algebraic introduction with geometric applications, 
Cambridge Studies in Advanced Mathematics 60, Cambridge University Press, Cambridge, 1998

\bibitem{GN}
S. Goto and K. Nishida, 
The Cohen-Macaulay and Gorenstein Rees algebras associated to filtrations, 
Mem. Amer. Math. Soc. 526 (1994)

\bibitem{TaiHa}
H. T. H\`{a}, 
Multigraded regularity, $a^{\ast}$-invariant and the minimal free resolution, 
J. Algebra 310 (2007), 156--179

\bibitem{Ha}
F. Hayasaka, 
Asymptotic periodicity of grade associated to multigraded modules, 
Proc. Amer. Math. Soc. 140 (2011), 2279--2284

\bibitem{HHR}
M. Herrmann, E. Hyry and J. Ribbe, 
On the Cohen-Macaulay and Gorenstein properties of multigraded Rees algebras, 
Manuscripta Math. 79 (1993), 343--377

\bibitem{Ho}
M. Hochster, 
Rings of invariants of tori, Cohen-Macaulay rings generated by monomials, and polytopes, 
Ann. of Math. 96 (1972), 318--337


\bibitem{HS}
C. Huneke and J. D. Sally, 
Birational extensions in dimension two and integrally closed ideals, 
J. Algebra 115 (1988), 481--500

\bibitem{HSw}
C. Huneke and I. Swanson, 
Integral closure of ideals, rings, and modules, 
London Mathematical Society Lecture Note Series, vol. 336, Cambridge University Press, Cambridge, 2006


\bibitem{Hy}
E. Hyry, 
The diagonal subring and the Cohen-Macaulay property of a multigraded ring, 
Trans. Amer. Math. Soc. 351 (1999), 2213--2232

\bibitem{KR}
D. Kirby and D. Rees, 
Multiplicities in graded rings I : The general theory, 
Contemporary Math. 159 (1994), 209--267

\bibitem{LT}
J. Lipman and B. Teissier, 
Pseudo-rational local rings and a theorem of Brian\c{c}on-Skoda about integral closures of ideals, 
Michigan Math. J. 28 (1981), 97--116


\bibitem{R1961}
D. Rees, 
A note on analytically unramified local rings, 
J. London Math. Soc. 36 (1961), 24--28


\bibitem{R}
D. Rees, 
Generalizations of reductions and mixed multiplicities, 
J. London Math. Soc. (2) 29 (1984), 397--414




\bibitem{RRV}
L. Reid, L. G. Roberts and M. A. Vitulli, 
Some results on normal homogeneous ideals, 
Comm. Algebra 31 (2003), 4485--4506

\bibitem{SV}
P. Sarkar and J. K. Verma, 
Local cohomology of multi-Rees algebras, joint reduction numbers and product of complete ideals, 
Nagoya Math. J. 228 (2017), 1--20

\bibitem{S}
P. Singla, 
Minimal monomial reductions and the reduced fiber ring of an extremal ideal, 
Illinois J. Math. Volume 51, Number 4 (2007), 1085--1102


\bibitem{Tr}
N. V. Trung, 
Positivity of mixed multiplicities, 
Math. Ann. 319 (2001), 33--63

\bibitem{ZS}
O. Zariski and P. Samuel, 
Commutative Algebra. II, 
Graduate Texts in Mathematics, 29, Springer, 1975




\end{thebibliography}
\end{document}